\documentclass[10pt,a4paper]{article}
\usepackage[utf8]{inputenc}
\usepackage{amsfonts,amsmath,amsthm,amscd,amssymb,latexsym}
\usepackage[margin=1.2in]{geometry}
\usepackage{tikz}
\usepackage{tabularx,ragged2e,booktabs,caption}
\usepackage{lscape} 
\usepackage{makeidx}				
\usepackage{enumerate}               
\usepackage{titlesec}
\usepackage{mathabx}	
\usepackage{mathtools}
\usepackage{pifont}
\usepackage{color}
\usepackage{makeidx}				
\usepackage{appendix} 
\usepackage{longtable}
\usepackage{tocloft}
\usepackage{setspace}

\DeclareMathOperator{\Aut}{Aut}

\DeclareMathOperator{\LHC}{\text{LHC}}
\DeclareMathOperator{\HS}{\text{HS}}

\DeclarePairedDelimiter\floor{\lfloor}{\rfloor}

\theoremstyle{plain}
\newtheorem{theorem}{Theorem}[section] 
\newtheorem{lemma}[theorem]{Lemma}
\newtheorem{proposition}[theorem]{Proposition}
\newtheorem{corollary}[theorem]{Corollary}
\newtheorem{definition}[theorem]{Definition}

\theoremstyle{definition}
\newtheorem{example}[theorem]{Example}
\newtheorem{remark}[theorem]{Remark}
\newtheorem{fact}[theorem]{Fact}

\begin{document}

\title{Endomorphisms of Cuboidal Hamming Graphs, Latin~Hypercuboids of Class $r$, and Mixed MDS Codes}
\author{Artur Schaefer\\
  {\small Mathematical Institute, University of St Andrews}\\
  {\small North Haugh, St Andrews KY16 9SS, UK}\\
  {\small  as305@st-andrews.ac.uk}}
\date{}
\maketitle

\begin{abstract}
 In this paper we investigate the existence of singular endomorphisms of the cuboidal Hamming graph $H(n_1,...,n_d,S)$ over the set $\left[ n_1\right]\times \left[ n_2\right]\times \cdots \times \left[ n_d\right]$, where $\left[ n\right]=\{1,...,n\}$, which is a generalisation of the well-known (cubic) Hamming graph over $\left[ n\right]^{d}$. Two vertices in $H$ are adjacent, if their Hamming distance lies in the set $S$. In this paper $S=\{1,...,r\}$, for some integer $1\leq r\leq d-1$, and we first show that the singular endomorphisms of minimal rank ( which is the size of their image) of $H(n,...,n,S)$ correspond to Latin hypercubes of class $r$ (those were originally defined by Kishen \cite{kishen50}). Then we generalise those hypercubes to Latin hypercuboids of class $r$. We discuss the existence of these objects, provide constructions and count Latin hypercuboids for small parameters. In the last part, we extend the natural connection between Latin hypercubes of class $r$ and MDS codes to Latin hypercuboids of class $r$ leading to the definition of MDS codes for mixed codes (mixed MDS codes), that is for codes over hypercuboids. Here, we demonstrate the interdependence between graph endomorphisms, Latin hypercuboids and mixed MDS codes.
\end{abstract}

\section{Introduction}\label{section1}

The motivation to this research comes from synchronization theory; in particularly, the study of synchronizing (permutation) groups, where a group $G$ is called synchronizing, if for every map $f$ on $n$ points the semigroup $\langle G,f\rangle$ contains a constant map. Cameron and Kazanidis \cite{pjc08} proved that the study of synchronizing permutation groups is equivalent to the study of graph endomorphisms of graphs admitting a complete core (clique number equal to chromatic number), where a graph homomorphism is mapping edges to edges.

The Hamming graph $H(d,n)$ has a complete core. In general, a Hamming graph $H(d,n,S)$ is a graph whose vertices are tuples in $\left[ n\right]^{d}$ where two vertices are adjacent, if their Hamming distance is in the set $S$, for $S\subseteq \{1,...,d-1\}$. The Hamming graph occurs for $S=\{1\}$, but we simply write $H(d,n)$. In \cite{schaeferhamminggraph} it is shown that the singular endomorphisms of $H(d,n,S)$, for $S=\{1,...,r\}$, have ranks $n^{k}$, for some $k=r,...,d-1$. In particular, the singular endomorphisms of $H(d,n)$ of rank $n^{r}$ are induced by Latin hypercubes. A natural generalisation of $H(d,n,S)$ is the graph $H(n_1,...,n_d,S)$; its vertices are tuples in $\left[ n_1\right]\times \left[ n_2\right]\times \cdots \times \left[ n_d\right]$ and two vertices are adjacent if their Hamming distance is in the set $S\subseteq \{1,...,d-1\}$. This graph is called the cuboidal Hamming graph (with respect to $S$).

In this article, we check the existence of singular endomorphisms of cuboidal Hamming graphs $H(n_1,...,n_d,S)$, where $S=\{1,...,r\}$, by investigating the existence of Latin hypercuboids of class $r$. The more specific Latin hypercubes of class $r$ have been initially defined by Kishen \cite{kishen50} and pose a further generalisation of Latin squares and Latin cubes, whereas here we define the more general Latin hypercuboids of class $r$. But before we get to that, we introduce the graph theoretical background and establish that the endomorphisms of minimal rank of $H(d,n,S)$ correspond to Latin hypercubes of class $r$. This correspondence is then lifted to the cuboidal version in Section \ref{section3}, which is dedicated to the analysis of Latin hypercuboids of class $r$. Not only are we going to provide an existence condition or rather a bound on the parameters of Latin hypercuboids, but also two constructions and a table with the numbers of Latin hypercuboids for small parameters. Then, in Section \ref{section4}, which is a prequel to Section \ref{section5}, we introduce mixed codes (cf. \cite{brouwer98}). We generalise the Singleton, Hamming and Plotkin bounds to mixed codes and define mixed MDS codes. Ultimately, in Section \ref{section5} we establish the interdependence between the singular endomorphisms, Latin hypercuboids and mixed MDS codes. 

%

\section{Endomorphisms of cuboidal Hamming graphs}\label{section2}

Cuboidal Hamming graphs are generalising Hamming graphs in the same way that hypercuboids are generalising hypercubes. This research is essentially continuing the research in \cite{schaeferhamminggraph}. The question we are tackling here is: For which parameters $n_1,...,n_d$ and $r$ does the cuboidal Hamming graph $H(n_1,...,n_d,S)$ admit singular endomorphisms, where $S=\{1,..,r\}$?  

In the case of Hamming graphs, that is the cubic case with $n_1=\cdots =n_d=n$ and $r=1$, it is known that these graphs always attain singular endomorphisms (because their clique number and chromatic number are identical). Moreover, in \cite{schaeferhamminggraph} the author proved that a singular endomorphism has rank $n^{k}$, is uniform (each kernel class has the same size), and is induced by a combination of Latin hypercubes of dimensions $d,d-1,...,d-k+1$, where $1\leq k \leq d-1$.
 
However, the existence of singular endomorphisms is not obvious anymore for cuboidal Hamming graphs. First of all, if a cuboidal Hamming graph $X$ attains a singular endomorphism $f$, then the semigroup $\langle \Aut(X),f\rangle$ is not synchronizing \cite{pjc08}, and consequently, it contains endomorphisms of minimal rank (which is strictly greater than $1$). So, it is sufficient to focus on the existence of endomorphisms of minimal rank.

\begin{fact}
 If there is no singular endomorphism of minimal rank, then there are no singular endomorphisms at all.
\end{fact}

Regarding graphs with clique number $\omega$ equal to the chromatic number $\chi$ we obtain the following triviality.

\begin{lemma}
 If a graph $X$ has $\omega(X)=\chi(X)$, then there are singular endomorphisms (of minimal rank).
\end{lemma}

The problem with this lemma is that the converse does not hold, in general. A counter-example is given by the Cartesian product of two cyclic graphs $C_5$ which can be found in \cite{schaeferhulls}. Its singular endomorphisms of minimal rank correspond to certain Latin squares; however, this graph has clique number $2$ and chromatic number $3$. In the same paper, we see that one class of graphs for which the converse holds are graphs which are hulls. 

However, as already mentioned the situation is not as simple as for the Hamming graph, where $\omega=\chi$ is satisfied. For instance Table \ref{tablecountingLatinhypercuboids} indicates that the graph $H(3,3,2,2,2,\{1,2\})$ does not satisfy this equality, because it does not admit any singular endomorphism. But, let us consider the cubic case first, that is $H(d,n,S)$, for $S=\{1,...,r\}$.

First, we state the definition of Latin hypercubes of dimension $d$, order $n$ and class $r$. Our definition is based on Latin hypercubes of class $r$ given in \cite{ethier08}. 

\begin{definition}
 A $d$-dimensional Latin hypercube of order $n^{r}$ and class $r$ is an $n\times n\times \cdots \times n$ ($d$ times) array based on $n^{r}$ distinct symbols, each repeated $n^{d-r}$ times, such that each occurs exactly once in each $r$-subarray. We will write $\LHC(d,n,r)$ for such cubes. For $r=1$, these cubes coincide with the Latin hypercubes defined by McKay and Wanless \cite{mckay08}.
\end{definition}




\begin{proposition}[Lemma 6.4 \cite{schaeferhamminggraph}]
  The singular endomorphism of minimal rank of $H(d,n,\{1,...,r\})$ coincide with Latin hypercubes of dimension $d$ and class $r$.
\end{proposition}

It is well known that singular endomorphisms exist for all parameters, if $r=1$; however, Ethier et. al. \cite{ethier08} provided an existence condition for $r>1$ (which is going to be generalised in the next section) which shows that there are cases where no singular endomorphisms exist. But, for which parameters do singular endomorphism exist in the non-cubic case? Well, for the Hamming graph $H(d,n)$ a well-known existence condition is given by the product of clique and coclique number.

\begin{lemma}[Remark in \cite{pjc08}]\label{lemmaonindependentsets}
 If $X$ is a vertex-transitive graph on $n$ vertices with complete core ($\omega(X)=\chi(X)$), then $n=\omega(X)\cdot \alpha(X)$, where $\alpha(X)$ is the coclique number.
\end{lemma}

Indeed, cliques and cocliques will play an important role for the singular endomorphisms of $H(n_1,...,n_d,S)$, too, as will be shown in Section \ref{section5}. However, Lemma 6.4 in \cite{schaeferhamminggraph} is saying more. It is, actually, saying that endomorphisms of minimal rank $n_1\cdots n_k$ exist, if and only if Latin hypercuboids of the corresponding type and class do. Thus, the next section is dedicated to the analysis of these Latin hypercuboids.

\section{Latin Hypercuboids of Class r}\label{section3}

In this section, we will deal with Latin hypercuboids and figure out their existence and their numbers for various parameters. First, we define these objects and describe their symmetry and their equivalence classes. Then, we discuss the existence of Latin hypercuboids of class $1$ and highlight the difference to Latin hypercuboids of class $k\geq 2$ by imposing a necessary condition on the parameters. Afterwards, we provide two constructions of Latin hypercuboids of class $r$ and a table counting small Latin hypercuboids of small parameters and small class. 

\subsection{Definition and Symmetry}

The definition given in Section \ref{section2} covered the Latin hypercubes of class $r$; here, we are extending this definition to Latin hypercuboids. These hypercuboids are generalising the Latin hypercubes of class $r$, in the same way as Latin rectangles are generalising Latin squares. 



\begin{definition}
 Let $n_1\geq n_2\geq \cdots \geq n_d\geq 2$ be integers. A \emph{Latin hypercuboid of dimension $d$, type $(n_1,...,n_d)$, class $r$ and order $n$} is an $\left[n_1\right]\times \left[n_2\right] \times \cdots \times \left[n_d\right]$ array based on $n$ distinct symbols, such that the symbols in every $r$-dimensional subarray occur at most once. However, if $n=\prod\limits_{i=1}^{r} n_i$, then in every $r$-dimensional subarray with $n$ cells each symbol occurs exactly once and in any other $r$-dimensional subarray each symbol occurs at most once. We write $\LHC(n_1,...,n_d,r)$ for a Latin hypercuboid of this order.
\end{definition}

Note, in the literature the terms Latin boxes or Latin parallelepipeds are used (for $r=1$). Also, it is obvious that our definition covers Latin hypercubes of class $r$, initially defined by Kishen \cite{kishen50}, as well as Latin boxes (of class $1$) which are found in various research articles (cf. \cite{evans60,cruse74,denley}). (We would like to mention that we have not found any previous publication mentioning Latin boxes of class $r$ or similar objects.)

\begin{remark}
\begin{enumerate}
 \item Each symbol in $LHC(n_1,...,n_d,r)$ appears the same number of times.
 \item If we do not mention the order of a Latin hypercuboid, then it should be obvious from the context (usually it is $\prod\limits_{i=1}^{r} n_i$).
 \item A \emph{partial Latin hypercuboid} is a Latin hypercuboid in the sense above where not every cell contains a symbol. In some cases this might mean that its order $n$ is strictly greater than $\prod\limits_{i=1}^{r} n_i$.
\end{enumerate}
\end{remark}


\begin{example}\label{example11}
The following is an example of a Latin hypercuboid of dimension $3$, type $(3,2,2)$ and class $2$. This cuboid has the top layer $L^{1}$ and bottom layer $L^{2}$.
 \[
           L^{1}=\begin{pmatrix}
          1&2&3\\ 4&5&6
         \end{pmatrix}, \quad 
         L^{2}=\begin{pmatrix}
          5&6&4\\2&3&1
         \end{pmatrix}\]
A partial Latin hypercuboid is, for instance, the following cube $M$
\[
           M^{1}=\begin{pmatrix}
          \ast&3\\ 5&6
         \end{pmatrix}, \quad 
         M^{2}=\begin{pmatrix}
          6&4\\3&1
         \end{pmatrix}, \]
         with empty cell denoted by $\ast$.
\end{example}

A Latin cuboid of class $r$ can be identified with a subset of an $\left[n_1\right]\times \cdots \times \left[n_d\right] \times \left[n_{d+1}\right]$ array $A$, where $n_{d+1}=n$; thus, symmetries of $A$ can be applied to the set of Latin hypercuboids.

The direct product $S_{n_1}\times \cdots \times S_{n_{d+1}}$ of symmetric groups is acting on $A$ via its natural action. The orbits under this action are the \emph{isotopy classes}\index{isotopy classes of Latin hypercuboids} of Latin hypercuboids of this type. In addition, if we are given a cube instead of a cuboid the symmetric group $S_{d+1}$ acts on the coordinates, as well. However, since the $n_i$ need not to be equal, we need to adjust and restrict this action to a subgroup, say, $\widetilde{S_{d+1}}$. The orbits under the action of 
\[\left( S_{n_1}\times \cdots \times S_{n_{d+1}} \right) \rtimes \widetilde{S_{d+1}},\]
are the \emph{paratopy classes}\index{paratopy classes}.

However, a weaker symmetry break leading to more equivalence classes is given by semi-reduced Latin hypercuboids. A Latin hypercuboid of dimension $d$, type $(n_1,...,n_d)$, class $r$ and order $n=\prod\limits_{i=1}^{r} n_i$ is semi-reduced, if the $n$ entries in the first $r$-subarray are naturally ordered like $1,2,...,n$. Every Latin hypercuboid of class $r$ is similar to a semi-reduced one.

\subsection{An Existence Condition for Latin hypercuboids}

The fundamental difference between Latin hypercuboids of class $1$ and Latin hypercuboids of class $r\geq 2$ is that a hypercuboid does not exist for every choice of parameters. The next example shows that Latin hypercuboids of class $1$ exist for any set of parameters $(n_1,n_2,..., n_d)$, whereas the subsequent lemma indicates that Latin hypercuboids of class $r$ do not exist for small parameters. 

\begin{example}\label{examplelatinhypercuboidofclass1}
 Let $\mathbb{Z_{n_i}}$ be the integers modulo $n_i$. A Latin hypercuboid of class $1$ is given by the function 
 \[\mathbb{Z}_{n_1}\times \mathbb{Z}_{n_2}\times \cdots \times\mathbb{Z}_{n_d}\rightarrow \mathbb{Z}_{n_1}, (a_1,...,a_d)\mapsto \sum_{i=0}^{d} a_i.\]
 To check that this truly is a Latin hypercuboid, we need to pick a coordinate $i$ and fix all others, that is all entries are equal except for position $i$. Then, the sums are equal if and only if the entries in position $i$ are equal.
\end{example}

Moving from class $1$ to class $r\geq 2$ the existence is not guaranteed any more. More on this is given by the following lemma; its proof is a generalisation of the corresponding result on Latin hypercubes of class $r$ \cite[Lemma 6.1.1]{ethier08}.

%

\begin{lemma}
 Let $\LHC(n_1,...,n_d,r)$ be a Latin hypercuboid of class $r\geq 2$. Then, its parameters satisfy
 \begin{equation}\label{boundexistenceclass2}
  \sum\limits_{i=1}^{d} n_i-\prod\limits_{i=1}^{r}n_i \leq d-1.
 \end{equation}
\end{lemma}
\begin{proof}
 Consider the origin of the coordinate system and let $0$ be the common point of all the coordinate axes. Since the cuboid is of class $r\geq 2$, any pair of coordinate axes is not allowed to have another common point. In sum, we have $(n_1-1)+(n_2-1)+\cdots + (n_d-1)$ distinct symbols on the coordinate axes. However, there is a total of $n_1n_2\cdots n_r$ distinct symbols used for the hypercuboid. Hence, 
 \[1+\sum\limits_{i=1}^{d} (n_i-1) \leq \prod\limits_{i=1}^{r} n_i.\]
\end{proof}

\begin{table}[t!]
 \begin{center}
  \begin{tabular}{|c|c|c|}\hline
   $d=4$ & $d=5$& $d=6$ \\\hline
      
$[2,2,2,2]$ & $[2,2,2,2,2]$ & $[2,2,2,2,2,2]$\\
 & $[3,2,2,2,2]$ & $[3,2,2,2,2,2]$\\
 & $[3,3,3,3,2]$ & $[3,3,3,2,2,2]$\\
 & $[3,3,3,3,3]$ & $[3,3,3,3,2,2]$\\
 &  & $[3,3,3,3,3,2]$\\
 &  & $[3,3,3,3,3,3]$\\
 &  & $[4,2,2,2,2,2]$\\
 &  & $[4,3,3,3,3,2]$\\
 &  & $[4,4,4,4,4,2]$\\
 &  & $[4,3,3,3,3,3]$\\
 &  & $[4,4,4,4,3,3]$\\
 &  & $[4,4,4,4,4,3]$\\
 &  & $[4,4,4,4,4,4]$\\\hline
  \end{tabular}
 \end{center}
\caption{Parameters with $n_1\leq 5$ and $r=2$ not satisfying bound (\ref{boundexistenceclass2}).}\label{tablecuboidsofclass2}
\end{table}

Table \ref{tablecuboidsofclass2} contains all the parameters $n_1,...,n_d$, for $n_1\leq 5$ and $d\leq 6$, not satisfying the inequality. This, provides a non-existence argument for these parameters. 


\begin{corollary}
 \begin{enumerate}
  \item The parameters of $\LHC(d,n,r)$ need to satisfy $ d \leq \frac{n^{r}-1}{n-1}$.
  \item In particular, for $r=2$ it must hold $d\leq n+1$.
 \end{enumerate}
\end{corollary}

 Note, the bound for $\LHC(d,n,r)$ is not tight, in general. For instance, Ethier has shown that the parameters actually need to satisfy $d\leq (n-1)^{r-1}+r$ (cf. \cite[Thm. 6.1.2]{ethier08}), but our simple generalisation is good enough for counting purposes.

\subsection{Constructing Latin Hypercuboids}

\subsubsection{Construction 1: An Elementary Construction}

Next, we turn to two constructions which provide Latin hypercuboids of class $r$. The first one is well known for the cubic case case where $r=1$ and $d=3$ \cite{mckay08,kuhl11} and for $r=2$ and some $d$ \cite{saxena60};  however, we show this construction can be generalised to higher classes, too, (including some types of Latin hypercuboids). Afterwards, this construction is demonstrated on an example.

We present our construction in two steps. In the first step we create a Latin hypercube of class $r$ from a Latin square; then in the second step we construct a Latin hypercube of dimension $d+1$ from a Latin hypercube of dimension $d$.


\begin{lemma}\label{lemmaconstruction1}
Let $L$ be a $d$ dimensional $\left[n\right]\times \left[n\right]\times \cdots \times \left[n\right]$ array whose entries are the $d$-tuples over the set $\{1,...,n\}$ and $S$ an $n\times n$ Latin square. Then, we can construct a Latin hypercube $\LHC(d+1,n,r)$. 
\end{lemma}
\begin{proof}
 Each row of $S$ corresponds to a permutation $\phi_i$ in the symmetric group $ S_{n}$, for $i=1,...,n$. To obtain the $i$th layer in the new Latin hypercuboid we apply $\phi_i$ to the entries of $L$ via $(x_1,...,x_d){\phi_i}=(x_1{\phi_i},...,x_d{\phi_i})$. It can easily be checked that this construction works.
\end{proof}

\begin{corollary}
 For $n\geq 2$ and $r\geq 1$, there always exists a Latin hypercube \linebreak[4]$\LHC(r+1,n,r)$.
\end{corollary}


The previous construction is straightforward to extend. Assume $S$ is a Latin hypercube $\LHC(r+1,n,r)$ whose $r$-layers correspond to permutations $\phi_i \in S_N$, where $N=n^{r}$ and $i=1,...,n$. If $L$ is a Latin hypercube $\LHC(d,n,r)$, then $L{\phi_i}$ is the $i$th layer of a Latin hypercube $\LHC(d+1,n,r)$ provided the following condition is satisfied:
\begin{equation}\label{condition1}
 \text{No }(r+1) \text{-layer in } L \text{ is constructed by applying the permutations } \phi_i \text{ to an } r \text{-layer.}
\end{equation}
Note that this condition implicitly identifies the $N$ symbols as $r$-tuples over $\{1,...,n\}$, and thus the $\phi_i$ correspond to permutations in $S_n$.

\begin{lemma}\label{lemmaconstruction2}
Let $L$ be a Latin hypercube $\LHC(d,n,r)$ and $S$ a Latin hypercube \linebreak[4]$\LHC(r+1,n,r)$. Then, we can embed $L$ into an Latin hypercube $\LHC(d+1,n,r)$, provided condition \ref{condition1} is satisfied.
\end{lemma}
\begin{proof}
Consider the layers of $S$ as permutations $\phi_i\in S_{N}$, where $N=n^{r}$ and $i=1,...,r+1$. Then, the $i$th layer of the new Latin hypercuboid is $L{\phi_i}$ (where $\phi_1$ is the identity).
\end{proof}

As can be seen from the next example this construction can be modified to create Latin hypercuboids, too.


\begin{example}
 \[L=\begin{pmatrix}
          (1,1)&(1,2)&(1,3)\\
          (2,1)&(2,2)&(2,3)\\
         \end{pmatrix}, \quad 
         S=\begin{pmatrix}
          1&2&3\\
          2&3&1
         \end{pmatrix}\]
The rows of $S$ provide us with permutations $\phi_1=1$ and $\phi_2=(1,2,3)$; moreover, we set $\psi=(1,2)$. Then, applying $\psi$ to the first coordinate of each entry and $\phi_i$ to the second gives us the entries of layer $i$. Thus we obtain the Latin cuboid $L^{\ast}$ given by the two layers
\[L^{\ast,1}=\begin{pmatrix}
          (1,1)&(1,2)&(1,3)\\
          (2,1)&(2,2)&(2,3)\\
         \end{pmatrix},L^{\ast,2}=\begin{pmatrix}
          (2,2)&(2,3)&(2,1)\\
          (1,2)&(1,3)&(1,1)\\
         \end{pmatrix}, \]
This cuboid is essentially the same as in Example \ref{example11}
\end{example}

Certainly, before applying this construction we would need to check whether the existence condition in (\ref{boundexistenceclass2}) or Ethier's bound $d\leq (n-1)^{r-1}+r$ is satisfied by the parameters. If so, then repeatedly applying Lemma \ref{lemmaconstruction2} leads to a finite chain of embeddings. In this regard, it would be interesting to know whether or not Ethier's bound provides a sufficient condition on the existence of Latin hypercubes of class $r$. In other words, is his bound strict? Unfortunately, this question is out of the scope of this research.

\subsubsection{Construction 2: Extending Quasigroups}

The second construction makes use of the notion of quasigroups. A \emph{quasigroup} is a set $Q$ with binary operation which admits the Latin square property, i.e., for all $a,b\in Q$ there exist unique elements $x,y\in Q$ such that the following equations hold
\begin{align*}
 ax&=b,\\
 ya&=b.
\end{align*}
Through the Latin square property quasigroups are equivalent to Latin squares (for more on quasigroups the reader is pointed to \cite{smith07}).

However, a natural generalisation form $d$-ary quasigroups. Such groups correspond to Latin hypercubes of dimension $d$ and class $1$. A \emph{$d$-ary quasigroup} is a map $f:Q^{d}\rightarrow Q$ such that the equation $f(x_1,...,x_d)=x_{d+1}$ can be uniquely solved for one of the variables, if the remaining $d$ variables are known. In this sense, a quasigroup from above is a binary ($2$-ary) quasigroup. Furthermore, an additional modification allows us to construct Latin hypercubes of dimension $d$ class $r$.

\begin{definition}
We call a map $f:Q^{d}\rightarrow Q^{r}$ a \emph{$d$-ary quasigroup of class $r$}, if the equation
\[f(x_1,...,x_d)=(x_{d+1},...,x_{d+r})\]
can be uniquely solved for any $r$ variables, if the remaining $d$ variables are known.
\end{definition}

A $d$-ary quasigroup of class $r$ can be interpreted as a Latin hypercube of dimension $d$ and class $r$ by considering the first $d$ coordinates $(x_1,...,x_d)$ as positions and the last $r$ coordinates as the entries. 

Obviously, such a map reminds us of linear maps and matrices. Hence, we provide the following construction. Let $Q$ be a finite field and $f$ a $d\times r$ matrix over $Q$. Moreover, let $e_1,...,e_d$ be the standard basis of the vector space $Q^{d}$ and define a $k$-dimensional layer $L$ (a $k$-layer) to be a subspace spanned by any choice of $k$ of the vectors $e_1,...,e_d$. 

\begin{proposition}
 A $d\times r$ matrix $f$ is a $d$-ary quasigroup of class $r$ (and thus a Latin hypercube of dimension $d$ and class $r$), if $f$ is injective on any $k$-layer.
\end{proposition}

\begin{example}
 Let $Q$ be the field $GF(3)$ and $f$ be the matrix
 \[
 \begin{pmatrix}
  1&0\\
  0&1\\
  1&1\\
  1&2
 \end{pmatrix}.
 \]
 Then, the $2$-layer $L$ spanned by $e_1$ and $e_4$ is mapped to $Lf$ (action on the right) which is the $2$-dimensional subspace spanned by the first and fourth row of $f$.
\end{example}

In a similar way it is possible to construct Latin hypercuboids instead of hypercubes. For instance, in Example \ref{examplelatinhypercuboidofclass1} we can identify the map with the $1\times d$ matrix consisting of $1$'s.

\subsection{Counting Latin Hypercuboids}

Latin squares have been counted for many decades, and so do Latin rectangles. More recently, McKay and Wanless \cite{mckay08} provided the numbers of Latin hypercubes of class $1$ for small dimensions. However, after a thorough research, we were not able to find any counting of Latin hypercuboids of class $r$ and not even the numbers of $3$-dimensional Latin cuboids of class $1$. Thus, in Table \ref{tablecountingLatinhypercuboids} we provide the numbers of Latin hypercuboids of dimension $d$, type $(n_1,...,n_d)$ and class $r$. 

The numbers appearing in Table \ref{tablecountingLatinhypercuboids} have been generated using the constraint satisfaction program MINION developed at the University of St.~Andrews. Each number represents the number of semi-reduced hypercuboids, and we provide the formula for the whole number of hypercuboids in a moment. But before, using Inequality \ref{boundexistenceclass2}, we were able to eliminate many small parameters indicated by $0^{1)}$. The 'minus' entries indicate the case $r\geq d$, where no hypercuboids can exists. Finally, a question mark indicates that we were not able to determine this number. 

\begin{table}

\begin{center}
{
\begin{tabular}{|c|l|c|c|c|}\hline
 & & \multicolumn{3}{c|}{Class}\\
Dimension & \multicolumn{1}{c|}{Type} & $r=1$ & $r=2$ & $r=3$ \\\hline\hline
3 & (2,2,2) & 1 & 1 & - \\
3 & (3,2,2) & 6 & 4 & - \\
3 & (3,3,2) & 4 & 448 & - \\
3 & (3,3,3) & 4 & 40 & - \\
3 & (4,2,2) & 321 & 81 & - \\
3 & (4,3,2) & 1,128 & 190,992 & - \\
3 & (4,4,2) & 792 & 3,089,972,673 & - \\
3 & (4,3,3) & 5,664 & 1,219,584 & - \\
3 & (4,4,3) & 2,304 & ? & - \\
3 & (4,4,4) & 2,304 & 10,123,306,543 & - \\
3 & (5,2,2) & 33,372 & 1,936 & - \\
3 & (5,3,2) & 2,118,624 & ? & - \\
3 & (5,4,2) & 5,360,352 & ? & - \\
3 & (5,5,2) & 2,288,256 & ? & - \\\hline
4 & (2,2,2,2) & 1 & $0^{1)}$ & 1 \\
4 & (3,2,2,2) & 38 & 0 & 11520 \\
4 & (3,3,2,2) & 12 & 176 & ? \\
4 & (3,3,3,2) & 8 & 104 & ? \\
4 & (3,3,3,3) & 8 & 104 & ? \\
4 & (4,2,2,2) & 119,001 & 576 & ? \\
4 & (4,3,2,2) & 526,824 & ? & ? \\
4 & (4,4,2,2) & 203,256 & ? & ? \\
4 & (4,3,3,2) & 4,335,648 & ? & ? \\
4 & (4,4,3,2) & 655,200 & ? & ? \\
4 & (4,4,4,2) & 515,808 & ? & ? \\
4 & (4,3,3,3) & 173,325,408 & ? & ? \\
4 & (4,4,3,3) & 3,998,880 & ? & ? \\
4 & (4,4,4,3) & 1,540,512 & ? & ? \\
4 & (4,4,4,4) & 1,540,512 & ? & ? \\\hline
5 & (2,2,2,2,2) & 1 & $0^{1)}$ & $0^{1)}$ \\
5 & (3,2,2,2,2) & 990 & $0^{1)}$ & ? \\
5 & (3,3,2,2,2) & 76 & 0 & ? \\
5 & (3,3,3,2,2) & 24 & 0 & ? \\
5 & (3,3,3,3,2) & 16 & $0^{1)}$ & ? \\
5 & (3,3,3,3,3) & 16 & $0^{1)}$ & ? \\\hline
6 & (2,2,2,2,2,2) & 1 & $0^{1)}$ & $0^{1)}$ \\
6 & (3,2,2,2,2,2) & 395,094 & $0^{1)}$ & $0^{1)}$ \\
6 & (3,3,2,2,2,2) & 1,980 & 0 & 0 \\
6 & (3,3,3,2,2,2) & 152 & $0^{1)}$ & ? \\
6 & (3,3,3,3,2,2) & 48 & $0^{1)}$ & ? \\
6 & (3,3,3,3,3,2) & 32 & $0^{1)}$ & ? \\
6 & (3,3,3,3,3,3) & 32 & $0^{1)}$ & ? \\\hline
7 & (2,2,2,2,2,2,2) & 1 & $0^{1)}$ & $0^{1)}$ \\
7 & (3,2,2,2,2,2,2) & ? & $0^{1)}$ & $0^{1)}$ \\
7 & (3,3,2,2,2,2,2) & 790,188 & $0^{1)}$ & ? \\
7 & (3,3,3,2,2,2,2) & 3,960 & $0^{1)}$ & ? \\
7 & (3,3,3,3,2,2,2) & 304 & $0^{1)}$ & ? \\
7 & (3,3,3,3,3,2,2) & 96 & $0^{1)}$ & ? \\
7 & (3,3,3,3,3,3,2) & 64 & $0^{1)}$ & ? \\
7 & (3,3,3,3,3,3,3) & 64 & $0^{1)}$ & ? \\\hline
\end{tabular}
}
\end{center}
 \caption{Counting Latin hypercuboids of Class $r$}\label{tablecountingLatinhypercuboids}
\end{table}

As mentioned above, the numbers given count semi-reduced Latin hypercuboids. When counting Latin hypercubes of dimension $d$, we are able to reduce the effort dramatically by normalising each of the $d$ coordinate axes (cf. McKay and Wanless \cite{mckay08}). However, it is not that simple for Latin hypercuboids of class $1$ and even more difficult for higher classes; but, we have still applied the most obvious symmetry break by normalising the first $r$-subarray, i.e. by counting semi-reduced Latin hypercuboids. The number of Latin hypercuboids is then given by the following product:
\[\LHC(n_1,...,n_d,r)=h_{(n_1,...,n_d,r)}\cdot c,\] 
where $c=\left(  \prod\limits_{i=1}^{r} n_i\right)!$ and $h_{(n_1,...,n_d,r)}$ is the number provided in the table.

\section{Mixed Codes}\label{section4}

In this section, we introduce mixed codes. Unlike the common codes which are over a fixed alphabet, these codes are codes over hypercuboids that is over various alphabets. We will introduce these codes as error-correcting codes, in the usual manner, provide the generalised Singleton bound, Hamming bound and Plotkin bound and define mixed MDS codes. There is only a few known about mixed codes and the reader might refer to \cite{brouwer98} and its references.

\subsection{The Definition of Mixed Codes}

\begin{definition}
 \begin{enumerate}
  \item An \emph{alphabet} $A$ is a finite set of symbols. If $|A|=n$, then it is an $n$-ary alphabet. We write $\left[n\right]$ for the alphabet $\{0,...,n-1\}$.
  \item A $d$-dimensional array $n_1\times n_2\times \cdots \times n_d$ is the Cartesian product $\left[n_1\right]\times \left[n_2\right]\times \cdots \times \left[n_d\right]$. Such an array forms a $d$-dimensional hypercuboid of type $(n_1,...,n_d)$ (sometimes it is useful to assume $n_1\geq n_2 \geq \cdots \geq n_d$).
  \item A \emph{cuboidal Hamming space}\index{cuboidal Hamming space} $\HS$ is a $d$-dimensional array and we write $\HS(n_1,...,n_d)$. The elements of $\HS$ are tuples $(x_1,x_2,...,x_d)$, for $x_i\in \underline{n_i}$ and for all $i=1,...,n$. They are sometimes called words $x_1x_2...x_d$ where $d$ is the length of the word.
  \item A \emph{mixed code}\index{mixed code} $C$ is a subset of $\HS(n_1,...,n_d)$. Codewords of length $d$ are elements of $C$.
 \end{enumerate}
\end{definition}

\begin{remark}
 The Hamming distance $d(v , w)$ between to $d$-tuples is the number of distinct positions in $v$ and $w$. It is a metric on $\HS$, and the weight of a codeword $v$ is defined as the distance $d(v,0)$, where $0$ is a $d$-tuple consisting of $0$'s (like for error-correcting codes).
\end{remark}

\begin{definition}
 A code $C$ is
 \begin{itemize}
  \item \emph{$t$-error-detecting}\index{$t$-error-detecting}, if $d(x,y)>t$, for all $x\neq y\in C$,
  \item \emph{$t$-error-correcting}\index{$t$-error-correcting}, if there do not exist distinct words $x,y\in C$ and $z\in \HS$ with $d(x,z)\leq t$ and $d(y,z)\leq t$.
 \end{itemize}
\end{definition}

\begin{definition}
 We define the set $d(C)$ to be $\{d(x,y): x\neq y \in C\}$. The minimum distance $\delta(C)$ is the minimum in $d(C)$.
\end{definition}

The following is obvious.

\begin{lemma}
 $C$ is $t$-error-correcting if and only if $\delta(C)>2t$.
\end{lemma}

\begin{definition}
 We say the code $C$ is an $\overline{n}$-ary $(d,M,\delta)$-code, if $C\subseteq \HS(n_1,...,n_d)$ with $\overline{n}=(n_1,...,n_d)$, $|C|=M$ and minimum distance $\delta$. In this regard, we call $C$ an $\overline{n}$-ary $(d,M,\delta)$-mixed-code.
\end{definition}

Let $H$ be the direct product of symmetric groups $S_{n_1}\times \cdots \times S_{n_d}$ and $K$ a subgroup of $ S_{d}$. Then, the semi-direct product $H\rtimes K$ acts on $\HS(n_1,...,n_d)$, where $H$ acts on the entries and $K$ permutes the coordinates. The group $G=\Aut(\HS(n_1,...,n_d))$ is of the form $H \rtimes K$, and we say two codes $C$ and $D$ in $\HS(n_1,...,n_d)$ are equivalent, if there is an element $g\in G$ such that $Cg=D$.

\begin{lemma}
 If $C$ is an additive mixed code (that is $v+w\in C$, for all $v,w\in C$), then $\delta(C)$ is the minimum weight of all codewords.
\end{lemma}

\subsection{The Main Problem in Coding Theory and Bounds}

Let $\HS(n_1,...,n_d)$ be a cuboidal Hamming space. By $A_{\overline{n}}(d,\delta)$ we denote the maximum $M$ such that there is an $\overline{n}$-ary $(d,M,\delta)$-code. Like in common coding theory, the main problem is to find the value $A_{\overline{n}}(d,\delta)$, for fixed $\overline{n},d$ and $\delta$.

\begin{theorem}
For $n=n_1=\cdots =n_d$ it holds
 \begin{enumerate}
  \item $A_{\overline{n}}(d,1)=n^{d}$ and 
  \item $A_{\overline{n}}(d,d)=n$.
 \end{enumerate}
However, if some of the $n_i$ are distinct then
\begin{enumerate}
 \item $A_{\overline{n}}(d,1)=\prod\limits_{i=1}^{d}n_i$ and
 \item $A_{\overline{n}}(d,d)=n_d$
\end{enumerate}
\end{theorem}

\begin{theorem}[Generalised Singleton Bound\index{Generalised Singleton Bound}]
 For $d,\delta \geq 1$ it holds 
 \[A_{\overline{n}}(d,\delta)\leq \prod\limits_{i=1}^{d-\delta+1} n_{d-i+1}=n_{\delta}\cdots n_d.\]
\end{theorem}
\begin{proof}
  Let $C$ be a code with maximal $M$. If we remove one of the coordinates (puncturing), say $x_i$, we obtain an $\overline{n'}$-ary $(d-1,M,\delta-1)$ code. Hence, $A_{\overline{n}}(d,\delta)=M\leq A_{\overline{n'}}(d-1,\delta-1)$. Iterating this for any choice of $\delta-1$ coordinates $\{n_{i_{1}},...,n_{i_{\delta-1}}\}$ gives us
  \[A_{\overline{n}}(d,\delta)\leq A_{\overline{n^{\ast}}}(d-\delta+1,1)=\prod\limits_{j=1}^{d-\delta+1} n_{i_{j}},\]
 for the corresponding tuple $\overline{n^{\ast}}$. The right hand side attains its minimum for $n_{\delta}\cdots n_d$.
\end{proof}


What is the number of words $y\in \HS(n_1,...,n_d)$ of distance $\delta$ from a fixed word $x$? Well, we need to pick $\delta$ coordinates, and for each coordinate $x_i$ one of its possible $(n_i-1)$ entries. Therefore, the number of words having distance $\delta $ from $x$ is
\[s(x,\delta)=\sum\limits_{n_{i_{1}},...,n_{i_{\delta}}\in N} (n_{i_{1}}-1)(n_{i_{2}}-1)\cdots (n_{i_{\delta}}-1),\]
where (let $N=\binom{\{d\}}{\delta}$ be the set of possible choices). However, this number does not depend on the choice of $x$.

\begin{definition}
 Let $S(x,t)=\{ y\in \HS: d(x,y)\leq t\}$ be the sphere with radius $t$ and centre $x$. Sometimes, we write $S(t)=|S(x,t)|$ for the size of the sphere, since it is independent of $x$.
\end{definition}

The next lemma is well-known from common coding theory.

\begin{lemma}
 The sphere $S(x,t)$ contains $\sum\limits_{i=1}^{t} s(x,i)$ points. In addition, a code $C$ is $t$-error-correcting if and only if for any distinct pair of codewords $x,y$ the spheres $S(x,t)$ and $S(y,t)$ are disjoint.
\end{lemma}

\begin{theorem}
 If $C$ is a $t$-error-correcting code, then
 \[|C|\leq \dfrac{\prod\limits_{i=1}^{d} n_i}{S(t)}.\]
\end{theorem}
\begin{proof}
 Since the spheres are disjoint, the contained codewords satisfy $|C|\cdot S(t)\leq \prod\limits_{i=1}^{d} n_i$.
\end{proof}

\begin{corollary}[Generalised Hamming Bound\index{Generalised Hamming Bound}]
 It holds (for $\overline{n},d,t>0$)
 \[A_{\overline{n}}(d,2t)\leq\dfrac{\prod\limits_{i=1}^{d} n_i}{S(t)}.\]
\end{corollary}

\begin{theorem}[Generalised Plotkin Bound\index{Generalised Plotkin Bound}]
 Let $C$ be an $\overline{n}$-ary $(d,M,\delta)$-mixed-code with $rd<\delta$, where $r=1-\sum\limits_{i=1}^{d} \left( dn_i \right)^{-1}$. Then, 
 \[M \leq \floor*{ \dfrac{\delta}{\delta-rd}}.\]
\end{theorem}
\begin{proof}
This proof is a generalisation of the Plotkin bound found in Huffman's book \cite[p. 58]{huffmanbook}.
 Let $C$ be such a code and define $S=\sum\limits_{x\in C}\sum\limits_{x\in C} d(x,y)$. We count $S$ in two ways. First, because $\delta\leq d(x,y)$, it follows $M(M-1)\delta\leq S$. Second, let $\mathbb{M}$ be a matrix whose rows are the codewords in $C$ and $n_{i,a}$ the number of times the character $a\in \underline{n_i}$ appears in column $i$. As $\sum\limits_{a\in \underline{n_{i}}}n_{i,a}=M$, for all $i=1,...,d$, we have
 \begin{align*}
 S&=\sum\limits_{i=1}^{d}\sum\limits_{a\in \underline{n_{i}}} n_{i,a}(M-n_{i,a})=dM^2 - \sum\limits_{i=1}^{d}\sum\limits_{a\in \underline{n_{i}}} n_{i,a}^{2}.
 \end{align*}
 Now, by the Cauchy-Schwarz inequality, it holds $\left( \sum\limits_{a\in \underline{n_{i}}} n_{i,a} \right)^{2}\leq n_i \sum\limits_{a\in \underline{n_{i}}} n_{i,a}^{2}$.
 Therefore, we obtain
 \begin{align*}
S&\leq dM^{2}-\sum\limits_{i=1}^{d}\sum\limits_{a\in \underline{n_{i}}} n_{i,a}^{2} \leq dM^{2}- \sum\limits_{i=1}^{d} n_i^{-1}\left( \sum\limits_{a\in \underline{n_{i}}} n_{i,a} \right)^{2}\\
  &= dM^{2}-\sum\limits_{i=1}^{d}n_i^{-1} M^{2} =M^{2}r.
 \end{align*}
  By the assumption $rd<\delta$, the hypothesis follows from $M(M-1)\delta\leq S\leq M^{2}r$.
\end{proof}

\begin{definition}
 A \emph{mixed maximum distance separable code}\index{mixed MDS code} (mixed MDS code) is a mixed code attaining the generalised Singelton bound.
\end{definition}

\section{Connecting Endomorphisms, Hypercuboids and Mixed Codes}\label{section5}

In this section we describe both the connection between endomorphisms and mixed codes and between Latin hypercuboids and mixed codes. We also show that the necessary existence condition on endomorphisms given in Lemma \ref{lemmaonindependentsets} translates into an existence condition for mixed MDS codes.

One big question in the theory of MDS-codes is the classification of MDS codes with regards to their parameters. That means, the goal is to find all the parameters for which MDS-codes exist. This problem has been known for a long time, however a recent contribution is given by the Kokkala et. al \cite{kokkala2015classification}. Another purpose of this section is to contribute to this classification by providing results on mixed MDS codes.

Let us start with the cubic versions. Regarding Lemma \ref{lemmaonindependentsets}, the first result reveals the equivalence between maximal cliques, MDS codes and Latin hypercubes.


\begin{lemma}\label{lemmamaxcliques}
 The following are equivalent.
 \begin{enumerate}
  \item A maximal clique of size $n^{d-r}$ in $H(d,n;\{r+1,...,d\})$.
  \item An $n$-ary $(d,n^{d-r},r+1)$ MDS code.
  \item A Latin hypercube $\LHC(d-r,n,r)$.
 \end{enumerate}
\end{lemma}
\begin{proof}
 Any two vertices in the maximal clique have Hamming distance $r+1$; thus, the clique satisfies the Singleton bound. Moreover, every such code provides a maximal clique. Now, pick an MDS code; we show that we obtain a Latin hypercube of class $r$. Because the Hamming distance between any two codewords is $r+1$, the first $d-r$ coordinates can be considered as positions and the last $r$ coordinates as symbols in a Latin hypercube of class $r$. Conversely, if given a Latin hypercube $\LHC(d-r,n,r)$ we identify the symbols with $r$-tuples. Thus the Latin hypercube corresponds to a set of $d$-tuples where any two tuples have Hamming distance $r+1$. Hence, we obtain an MDS code.
\end{proof}


%
%

Lemma \ref{lemmamaxcliques} can be easily extended to hypercuboids.


\begin{theorem}
 The following are equivalent.
 \begin{enumerate}
  \item A maximal clique of size $\prod\limits_{i=r+1}^{d} n_i$ in $H(n_1,...,n_d;S)$, for $S=\{r+1,...,d\}$.
  \item An $(n_1,...,n_d)$-ary $(d,\prod\limits_{i=r+1}^{d} n_i,r+1)$ mixed MDS code.
  \item A Latin hypercuboid $\LHC(n_1,...,n_d,r)$.
 \end{enumerate}
\end{theorem}

\begin{example}
 Consider the set $M$ of tuples $ \{ (1,1,1), (2,3,1), (3,2,1),(1,2,2),\linebreak[4]  (2,1,2), (3,3,2)\}$. This set $M$ forms a mixed MDS code over $HS(3,3,2)$ and a maximal clique in the cuboidal Hamming graph $H(3,3,2,\{2,3\})$. However, this is also the Latin rectangle 
 \[ \begin{pmatrix}
     1&3&2\\
     2&1&3
    \end{pmatrix},
\]
where we identify the coordinates as $(\text{symbol},\text{column},\text{row})$.
\end{example}

Now, we provide the relations between Latin hypercuboids of dimension $d$, type $(n_1,...,n_d)$ and class $r$ and mixed MDS codes.

\begin{lemma}
 We can construct an $(\prod\limits_{i=1}^{r} n_i,n_1,...,n_d)$-ary $(d+1,\prod\limits_{i=1}^{d} n_i,2)$ mixed MDS code from an $\LHC(n_1,...,n_d,r)$.
\end{lemma}
\begin{proof}
 Simply consider the Latin hypercuboids as a subset of an $(d+1)$-array whose entries are the $\prod\limits_{i=1}^{r} n_i$ distinct symbols in the first coordinate of the mixed code.
\end{proof}

\begin{corollary}
We can construct an $(n_1,...,n_r,n_1,...,n_d)$-ary $(d+r,\prod\limits_{i=1}^{d} n_i,r+1)$ mixed MDS code from an $\LHC(n_1,...,n_d,r)$.
\end{corollary}
\begin{proof}
 Follows from the previous lemma by taking $(n_1,...,n_r)$-ary tuples as entries.
\end{proof}

\begin{theorem}
 An $(n_1,...,n_r,n_1,...,n_d)$-ary $(d+r,\prod\limits_{i=1}^{d} n_i,r+1)$ mixed MDS code induces an $\LHC(n_1,...,n_d,r)$.
\end{theorem}
\begin{proof}
 Place the codewords in the $n_1\times \cdots \times n_d$ array. If two codewords would have the same last $d$ positions, then their distance would not be $r+1$, since there would be only $r$ coordinates left. Thus, the words fill out this array and the first $r$ coordinates can be regarded as symbols. Now, if two codewords provide the same symbol (both words have the same first $r$ coordinates), then they need to differ in $r+1$ position coordinates. Hence, they are not in the same $r$-subarray, and thus, they are satisfying the definition of a Latin hypercuboid of class $r$.
\end{proof}

Once again, we obtain the following statement.

\begin{corollary}
 Assuming the correct parameters: Latin hypercuboids are equivalent to mixed MDS codes.
\end{corollary}

\section{Conclusion}
 In this research, we discussed the existence of singular endomorphisms of the graph $H(n_1,...,n_d,S)$ of minimal rank, where $S=\{1,...,r\}$. We have shown that their existence is linked to the existence of the newly defined Latin hypercuboids of class $r$ and equivalently to mixed MDS codes. We provided a bound on the parameters of Latin hypercuboids, two constructions, their numbers for small parameters, and their links to mixed MDS codes.
 
 As the reader has already observed there is much more work to do. Not only are we missing many more numbers of small Latin hypercuboids (which necessitate more sophisticated algorithms), but rather a more detailed analysis of those hypercuboids, like it was done for hypercubes. Moreover, here we have only started to work on mixed MDS codes; certainly, there are many more questions which have already been answered for common MDS codes and are thus left to the more enthusiastic reader.

\end{document}